\documentclass[reqno]{amsart}

\input xy
\xyoption{all}
\usepackage{epsfig}
\usepackage{color}
\usepackage{amsthm}
\usepackage{amssymb}
\usepackage{amsmath}
\usepackage{amscd}
\usepackage{amsopn}
\usepackage{url}

\usepackage{subfig}
\usepackage{graphicx}
\usepackage{caption}
\usepackage{cancel}
 \usepackage[colorlinks]{hyperref} 
 \usepackage[figure,table]{hypcap}
\hypersetup{   bookmarksnumbered,   pdfstartview={FitH},   citecolor={blue},   linkcolor={blue}, urlcolor={black},   pdfpagemode={UseOutlines}}

\usepackage{color} 

\numberwithin{equation}{section}
\newtheorem {Theorem}{Theorem}
\numberwithin{Theorem}{section}

\newtheorem {Lemma}[Theorem]    {Lemma}

\newtheorem {Proposition}[Theorem]{Proposition}
\newtheorem {Corollary}[Theorem]{Corollary}

\theoremstyle{definition}
\newtheorem{Definition}[Theorem]{Definition}
\theoremstyle{remark}
\newtheorem{Remark}[Theorem]{Remark}

\numberwithin{equation}{section}

\expandafter\chardef\csname pre amssym.def at\endcsname=\the\catcode`\@
\catcode`\@=11
\def\undefine#1{\let#1\undefined}
\def\newsymbol#1#2#3#4#5{\let\next@\relax
 \ifnum#2=\@ne\let\next@\msafam@\else
 \ifnum#2=\tw@\let\next@\msbfam@\fi\fi
 \mathchardef#1="#3\next@#4#5}
\def\mathhexbox@#1#2#3{\relax
 \ifmmode\mathpalette{}{\m@th\mathchar"#1#2#3}%
 \else\leavevmode\hbox{$\m@th\mathchar"#1#2#3$}\fi}
\def\hexnumber@#1{\ifcase#1 0\or 1\or 2\or 3\or 4\or 5\or 6\or 7\or 8\or
 9\or A\or B\or C\or D\or E\or F\fi}

\font\teneufm=eufm10
\font\seveneufm=eufm7
\font\fiveeufm=eufm5
\newfam\eufmfam
\textfont\eufmfam=\teneufm
\scriptfont\eufmfam=\seveneufm
\scriptscriptfont\eufmfam=\fiveeufm

\catcode`\@=\csname pre amssym.def at\endcsname

\newcommand{\Aa}{{\mathcal A}}

\newcommand{\Mm}{{\mathcal M}}
\newcommand{\no} {\noindent}

\newcommand {\pa} {\partial}
\newcommand{\Pp}{{\mathcal P}}

\newcommand {\rest}[1] {\big|_{#1}}

\newcommand{\Ss}{{\mathcal S}}

\def    \F      {{\mathbb F}}

\def    \C      {{\mathbb C}}
\def    \R      {{\mathbb R}}

\def    \Z      {{\mathbb Z}}

\def    \Q      {{\mathbb Q}}

\def    \12    {{\frac{1}{2}}}

\begin{document}


\setlength{\smallskipamount}{6pt}
\setlength{\medskipamount}{10pt}
\setlength{\bigskipamount}{16pt}





\title[Action Selectors and the Fixed Point Set]{Action Selectors and the
Fixed Point Set of a Hamiltonian Diffeomorphism}

\author[Wyatt Howard]{Wyatt Howard}

\address{Department of Mathematics, U.C. Santa Cruz,
Santa Cruz, CA 95064, USA}
\email{whoward@ucsc.edu}


\date{\today}


\begin{abstract}
In this paper we study the size of the fixed point set of a Hamiltonian diffeomorphism on a closed symplectic manifold which is both rational and weakly monotone.
We show that there exists a non-trivial cycle of fixed points whenever the action spectrum is smaller, in a certain sense, than required by the Ljusternik-Schirelman theory.
For instance, in the aspherical case, we prove that when the number of points in the action spectrum is less than or equal to the cup length of the manifold, then the cohomology of the fixed point
set must be non-trivial.  This is a consequence of a more general result that is applicable to all weakly monotone manifolds asserting that the same is true when the action selectors are related by an
equality of the Ljusternik-Schirelman theory.
\end{abstract}

\maketitle

\tableofcontents

\section{Introduction and main results}
\label{sec:intro}
\

In this paper we study the relationship between action selectors and the fixed point set for a Hamiltonian diffeomorphism defined on a symplectic manifold that is closed, rational, and weakly monotone.  We
are interested in understanding the size of the fixed point set for a time dependent Hamiltonian
whose action selectors satisfy a specific condition.
We use the Arnold Conjecture as a starting point for the statement of the main results of this paper.
The \textit{Arnold Conjecture} states that every Hamiltonian diffeomorphism $\phi_{H}$ of a compact symplectic manifold $(M, \omega)$ possesses at least as many fixed points as a function $f: M \rightarrow \R$ possesses critical points.  The weaker form of this conjecture asserts that the number of fixed points for $\phi_{H}$ is bounded below
by the cuplength of the manifold plus one, i.e.
$\# Fix(\phi_{H}) \geq CL(M) + 1$.
The \textit{$\F$-cuplength of $M$}, denoted $CL(M)$, of a topological space $M$ is the maximal integer $k$ such that there exists
classes $\alpha_{1}, \cdots , \alpha_{k}$ in the cohomology ring $H^{\ast > 0}(M; \F)$ satisfying
$$\alpha_{1} \cup \cdots \cup \alpha_{k} \neq 0.$$

While the Arnold Conjecture is still an open problem in the case when $M$ is a general rational, weakly monotone manifold, it has been proven in the symplectically aspherical case (\cite{floerCuplength}, \cite{hoferLjusternik}).
Suppose for the moment that $M$ is symplectically aspherical.
As is well known, one can use the basic properties and results concerning action selectors to prove the Arnold Conjecture when the Hamiltonian diffeomorphism has isolated fixed points, see e.g. \cite{ginzburgAction} and the references therein.  This is accomplished by using the spectrality properties of action selectors, meaning $c^{\alpha}(H) \in \Ss (H)$ where $\alpha \in H^{\ast}(M)$, $H$ is a Hamiltonian, $c^{\alpha}(H)$ denotes our action selector, and $\Ss (H)$ the action spectrum.  Using this fact, one is able to establish the following bound on the size of $\Ss (H)$:
$$\# \Ss(H) \geq CL(M) + 1 ,$$
\no which in turn implies $\# Fix(\phi_{H}) \geq CL(M) + 1$.  Now, when $H$ instead satisfies the condition
$\# \Ss (H) < CL(M) + 1$, it necessarily implies that the fixed point set
for $\phi_{H}$ cannot be isolated.  As a result, this presents us with the following question:
``How large'' is the set $Fix(\phi_{H})$ when $\# \Ss (H) < CL(M) + 1$ ?
\

\no This leads us to one of the main results of the paper.

\begin{Theorem}
\label{thm:aspherical}
Suppose that $(M, \omega)$ is a symplectic manifold, which is closed and symplectically aspherical
and $H$ is a time dependent Hamiltonian
with the property $\# \mathcal{S}(H) < CL(M) + 1$.  Let $F$ denote the set of fixed points for the Hamiltonian
diffeomorphism $\phi_{H}$.
Then $H^{j}(F) \neq 0$ for some $1 \leq j \leq 2n$.
\end{Theorem}

Theorem ~\ref{thm:aspherical} will actually become an almost immediate corollary once the following result has
been shown.

\begin{Theorem}
\label{thm:weakly}
Let $(M, \omega)$ be a closed, rational, and weakly monotone symplectic manifold, $H$ be a time dependent Hamiltonian that is one-periodic in time and define $F$ to be the fixed point set for the Hamiltonian diffeomorphism $\phi_{H}$.  Also assume that there exists cohomology elements $\alpha, \beta \in HQ^{\ast}(M)$, with $\alpha \neq 0$,
$\beta = \sum_{A} \beta_{A} e^{A}$ with each $deg( \beta_{A}) > 0$ and satisfying the condition $c^{\alpha \ast \beta}(H) = c^{\alpha}(H) - I^{c}_{\omega}(\beta)$.
\

\no Then $\beta \rest{ F} \neq 0$ in $HQ^{\ast}(F)$ and this implies $H^{k}(F) \neq 0$ for some $k > 0$.
\end{Theorem}

\no In the above theorem, we take $HQ^{\ast}(F) := H^{\ast}(F) \otimes \Lambda ^{\uparrow}_{\omega}$ and $I^{c}_{\omega}(\beta) := min \{ - \int _{A} \omega | \beta _{A} \neq 0 \}$; in Section \ref{sec:quantum} we explain this notation in more detail.

Now, when the symplectic manifold $M$ is symplectically aspherical the quantum cohomology groups, denoted above by $HQ^{\ast}(M)$, simply reduce to the usual cohomology groups.  As a result of this, the quantum product,
denoted by ``$\ast$'' above, reduces to being the cup product.  This implies that once we prove Theorem ~\ref{thm:weakly} we will end up with the following corollary:

\begin{Corollary}
\label{cor:aspherical}
Let $(M, \omega)$ be a closed symplectically aspherical manifold, $H$ be a time dependent Hamiltonian that is one-periodic in time and define $F$ to be the fixed point set for the Hamiltonian diffeomorphism $\phi_{H}$.  Suppose that there exists cohomology elements $\alpha, \beta \in H^{\ast}(M)$, with $\alpha \neq 0$ and $deg(\beta) > 0$ and satisfying the condition
$c^{\alpha \cup \beta}(H) = c^{\alpha}(H)$.
\

\no Then $\beta \rest{ F} \neq 0$ in $H^{\ast}(F)$ and implies that $H^{k}(F) \neq 0$ for $k = deg(\beta)$.
\end{Corollary}

We would like to point out the similarity of Theorem ~\ref{thm:weakly} to a result due to Viterbo.
In \cite{viterboRemarks} he deals with the Morse theoretic analogue of action selectors known
as \textit{critical value selectors} defined by the equation
$$c^{\alpha}_{LS}(f) = inf\{ a \in \R | \alpha \neq 0 \, \text{in} \, H^{\ast}(M^{a}) \},$$
\no where $\alpha \in H^{\ast}(M)$, $M^{a} = \{ x \in M | f(x) \leq a \}$, and $f: M \rightarrow \R$ is at least $C^{1}$.
Viterbo looks at the connection between the critical points of the function $f$ and the
the critical value selectors.  He establishes that when $M$ is a Hilbert manifold, $f$ a $C^{1}$-function on $M$ satisfying the Palais-Smale condition, and
$\alpha, \beta \in H^{\ast}(M)$ with cup-product $\alpha \cup \beta \neq 0$ in $H^{\ast}(M^{a})$,
then $c^{\alpha \cup \beta}_{LS}(f) \leq c^{\alpha}_{LS}(f)$.
When the critical value selectors satisfy the equality $c^{\alpha \cup \beta}_{LS}(f) = c^{\alpha}_{LS}(f)$, then $\beta$ is nonzero on $H^{\ast}(F_{a})$, where $F_{a}$ is the
set of critical points of $f$ at level $a = c^{\alpha}_{LS}(f)$.  As a result, $dim(F_{a}) \geq deg(\beta)$
and hence $F_{a}$ is uncountable when $deg(\beta) \neq 0$.

\subsubsection{Organization of the paper}
In Section \ref{sec:prelimin} we discuss our notational conventions and relevant definitions.  Within Section \ref{sec:prelimin} we have included several subsections where we outline the various tools and basic results concerning them.  These subsections are meant to highlight the important features that will be used to
prove Theorem ~\ref{thm:aspherical} and Theorem ~\ref{thm:weakly}.  We do however make sure to point out useful references in order to aid the reader who is concerned with understanding their finer details.  Then, in
Section \ref{sec:proofs}, we provide the proofs to Theorem ~\ref{thm:aspherical} and Theorem ~\ref{thm:weakly}.

\subsection{Acknowledgments}
The author is very grateful to Viktor Ginzburg for posing the problem and for numerous useful discussions.
The author would also like to thank Richard Montgomery, Marta Bator\'eo, Yusuf G\"oren, Doris Hein, and Gabriel Martins for useful discussions.

\section{Preliminaries}
\label{sec:prelimin}

\subsection{Conventions and basic definitions}
The objective for this section of the paper is to set notation, definitions, and tools such as, filtered Floer homology, filtered Floer cohomology, quantum cohomology, the basics of the Ljusternik-Schnirelman theory, and Alexander-Spanier cohomology.

\subsubsection{Symplectic manifolds}
Throughout the paper we will assume that $(M, \omega)$ is a closed symplectic manifold, i.e. $M$ is compact and
$\pa M = \emptyset$.  The manifold $M$ is \textit{monotone} if
$[\omega] \rest{\pi_{2}(M)} = \lambda c_{1}(M) \rest{\pi_{2}(M)}$ for some non-negative constant $\lambda$, where $[\omega] (A)$ and $<c_{1}(M), A >$ denotes the integral of the symplectic
form and the first Chern class over the cycle $A \in \pi_{2}(M)$ respectively.    A
\textit{negative monotone} manifold satisfies the same condition, but with $\lambda \leq 0$.  The manifold $M$ is
\textit{rational} if $<[\omega], \pi_{2}(M)> = \lambda_{0} \Z$, where $\lambda_{0} \geq 0$.  Let $N$ be the positive generator of the discrete subgroup $<c_{1}(M), \pi_{2}(M)>$ of $\R$.  We call
$N$ the \textit{minimal Chern number}.  A symplectic manifold $M$ is said to be \textit{weakly monotone} if it is monotone or $N \geq n-2$, where $dim(M) = 2n$, which also includes when
$c_{1}(M)\rest{\pi_{2}(M)} = 0$.  When $M$ has the property $[\omega] \rest{\pi_{2}(M)} = 0 = c_{1}(M) \rest{\pi_{2}(M)}$, then $M$ is called \textit{symplectically aspherical}.
\

In this paper we will be working with time dependent Hamiltonians $H$.
More specifically, we are going to be dealing with Hamiltonians which are one-periodic in time, meaning
$H: S^{1} \times M  \rightarrow \R$ with $S^{1} = \R / \Z$ and $H_{t}(\cdot) = H(t, \cdot)$ for $t \in S^{1}$.  Let $X_{H}$ denote the time dependent vector field
that $H$ generates, where $X_{H}$ satisfies $i_{X_{H}}\omega = -dH$.
Let $\phi^{t}_{H}$ denote the time dependent flow for the vector field $X_{H}$.
In this paper we are interested in studying the
time-one map of $\phi^{t}_{H}$.  We call the map $\phi_{H} := \phi^{1}_{H}$ a \textit{Hamiltonian diffeomorphism}.
\

Let $K$ and $H$ be time dependent Hamiltonians, then we define $(K \# H)_{t} := K_{t} + H_{t} \circ (\phi^{t}_{K})^{-1}$.  The flow for the time dependent vector field generated by the Hamiltonian $K \# H$ is
the composition $\phi^{t}_{K} \circ \phi^{t}_{H}$.  As an aside, the composition $K \# H$ may not necessarily be
one-periodic in time.  If, however, $H_{0} = 0 = H_{1}$, then the composition is one-periodic.  One is able to impose
this condition on $H$ by reparametrizing $H$ as a function of time without changing its time-one map.  This allows us to treat $K \# H$ as a one-periodic Hamiltonian.

\subsection{Filtered Floer homology and filtered Floer cohomology}
\label{sec:Floer}
\subsubsection{Capped periodic orbits and filtered Floer homology}

In this section we begin by introducing the basics of Floer homology.  We plan on only presenting the basic elements of Floer homology.  For a more in depth discussion and for more on the specific details we refer the reader to
\cite{mcduffJHolomorphic}, \cite{hoferSymplectic}, \cite{banyagaMorse}.
\

We start by looking at the contractible loops $x: S^{1} \rightarrow M$.  Since $x$ is contractible we can attach a disk along the the boundary of the loop, which produces a new mapping $u: D^{2} \rightarrow M$ with
$u \rest{S^{1}} (t) = x(t)$.  We call the map $u$ a \textit{capping} of the loop $x$ and use the notation $\bar{x}$ to represent the pair $(x,u)$.  Let $u_{1}$ and $u_{2}$ be two cappings for the loop $x$.  The two cappings
are equivalent if the integrals of $\omega$ and $c_{1}(M)$ over the sphere formed by the connected sum
$u_{1} \# (-u_{2})$ is equal to zero.  In the symplectically aspherical case all cappings of a fixed loop $x$ are equivalent.  Let $\mathcal{P}(H)$ be the set of contractible one-periodic solutions to $X_{H}$ and  $\bar{\mathcal{P}}(H)$ be the set of contractible capped one-periodic solutions to $X_{H}$.
\

The cappings of these loops allows us to define the \textit{action functional} $\Aa_{H}$ for a time dependent Hamiltonian $H$.
For a capped loop $\bar{x} = (x,u)$ we define
$$\Aa_{H}(\bar{x}) = - \int_{u} \omega +  \int^{1}_{0} H_{t}(x(t))dt .$$

\no The critical points for the action functional are the equivalence classes of capped loops $\bar{x}$ which
are one-periodic solutions to the equation $\dot{x}(t) = X_{H}(t, x(t))$.  The set of critical values for the action
functional is called the \textit{action spectrum} of $H$ and is denoted by $\Ss (H)$.  The action spectrum is a set of measure zero.  In addition, when the manifold $M$ is rational, $\Ss (H)$ is a closed set and implies that $\Ss(H)$ is also a nowhere dense set (\cite{hoferSymplectic}).
\

Following the terminology used in \cite{salamonMorse}, we will call a capped one-periodic orbit $\bar{x}$ of $H$ \textit{non-degenerate} if the pushforward $d \phi_{H}: T_{x(0)}M \rightarrow T_{x(0)}M$ has no eigenvalues equal to one.  When all of the one-periodic orbits of $H$ are
non-degenerate, then we say $H$ is \textit{non-degenerate}.  Note that the condition of degeneracy does not depend on
the capping of the loop $x(t)$.
\

By fixing a field $\F$ (i.e. $\Z _{2}, \Q,$ or $\C$) we can use the Conley-Zehnder index, denoted
$\mu_{CZ}$, to impose a grading on the vector space that is generated by the elements in the set $\bar{\Pp}(H)$ over $\F$.
Define $CF^{(-\infty, b)}_{k}(H)$, for $b \in ( -\infty, \infty ]$ and $b$ not an element in the set $\Ss(H)$, to be the vector space of sums given by
$$\sum_{\bar{x} \in \bar{\Pp}(H)} a_{\bar{x}} \bar{x}, $$
\no with $a_{\bar{x}} \in \F$, $\mu_{CZ}(\bar{x}) = k$, $\Aa_{H}(\bar{x}) < b$, and the number of terms in the sum
with $a_{\bar{x}} \neq 0$ is semi-finite, meaning for every $c \in \R$ the number of terms with $a_{\bar{x}} \neq 0$ and $\mathcal{A}_{H}(\bar{x}) > c$ is finite.  There is a linear boundary operator
$\pa : CF^{(-\infty, b)}_{k}(H) \rightarrow CF^{(- \infty, b)}_{k-1}(H)$, where for $\bar{x} \in \bar{\Pp}(H)$ with
$\mu_{CZ}(\bar{x}) = k$ is defined to be
$$\pa \bar{x} = \sum_{\mu_{CZ}(\bar{y}) = k - 1} n(\bar{x}, \bar{y}) \bar{y} $$
\no and $\pa ^{2} = 0$.  When $\F = \Z _{2}$ the number $n(\bar{x}, \bar{y})$ counts the number of components in the $1$-dimensional moduli space $\Mm (\bar{x}, \bar{y})$ $\, mod \, \, 2$.  For a more general field $\F$, the number $n(\bar{x}, \bar{y})$ is a bit more involved to describe and we refer the reader to \cite{floerCoherent}.
One can further define $CF^{(a,b)}_{k}(H) :=  CF^{ (-\infty, b) }_{k}(H) / CF^{ (-\infty, a) }_{k}(H)$, for
$- \infty \leq a < b \leq \infty$ not in $\Ss(H)$.
The above construction results in what is known as the \textit{filtered Floer homology} of $H$ and is denoted by
$HF^{ (a,b) }_{\ast}(H)$.  Note when $(a,b) = (- \infty, \infty)$ we end up with the standard Floer
homology $HF_{\ast}(H)$.
\

Since the results of this paper deal with Hamiltonians that are degenerate, it is worth pointing out that filtered Floer homology can be defined in the degenerate case.  Take $H$ to be a Hamiltonian on $M$ with
$a,b \not \in \Ss (H)$ and $M$ to be a rational manifold.  By virtue of the fact that we can always find a non-degenerate Hamiltonian $\tilde{H}$ from an arbitrarily small perturbation of $H$ it allows us to define
$$HF^{(a,b)}_{\ast}(H) = HF^{(a,b)}_{\ast}(\tilde{H}) .$$

\subsubsection{Filtered Floer cohomology}

Now that the basics of Floer homology have been presented it then becomes a fairly straightforward process to explain the setup for the Floer cohomology.
\

We again take $H$ to be a non-degenerate Hamiltonian and define $CF^{k}_{(b, \infty )}(H)$, for $b \in [- \infty , \infty )$ with $b$ not an element in $\mathcal{S}(H)$, to be the \textit{filtered cochain complex}.  We take  $CF^{k}_{( b, \infty )}(H)$ to be the set of formal sums
$$ \sum_{\bar{x} \in \bar{\mathcal{P}}(H) } \alpha_{\bar{x}} \bar{x} $$
\no with $\alpha _{\bar{x}} \in \F$, $\mu _{CZ}(\bar{x}) = k$, $\mathcal{A}_{H}(\bar{x}) > b$, and satisfies the finiteness condition that for every $c \in \R$ the number of terms with $\alpha_{\bar{x}} \neq 0$ and
$\mathcal{A}_{H}(\bar{x}) < c$ is finite.
Also, using the same numbers $n(\bar{x}, \bar{y})$ from the Floer chain complex determines a linear coboundary operator
$\delta: CF^{k}_{( b, \infty)}(H) \rightarrow CF^{k + 1}_{( b, \infty)}(H)$, given by
$$\delta \bar{x} = \sum_{\mu_{CZ}(\bar{y}) = k + 1} n(\bar{x}, \bar{y}) \bar{y},$$
\no where $\bar{x} \in \bar{\Pp}(H)$, $\mu_{CZ}(\bar{x}) = k$, and satisfies $\delta ^{2} = 0$.
We define $CF^{k}_{(a,b)}(H) := CF^{k}_{( a, \infty )}(H)/ CF^{k}_{(b, \infty)}(H)$, for $- \infty \leq a < b \leq \infty$ which are not elements of $\Ss (H)$.  This results in giving us the \textit{filtered Floer cohomology} of $H$ and is denoted by $HF^{\ast}_{(a,b)}(H)$.
\

Just like the case of Floer homology,
we can also define the filtered Floer cohomology for a degenerate Hamiltonian $H$ by choosing a non-degenerate Hamiltonian $\tilde{H}$ that is close to $H$ and setting
$$ H^{\ast}_{(a,b)}(H) = H^{\ast}_{(a,b)}(\tilde{H}). $$

\subsection{Quantum homology and quantum cohomology}
\label{sec:quantum}
We begin this section by assuming that our symplectic manifold $M$ is both weakly monotone and rational.
Define $I_{\omega}(A) = - \int _{A} \omega$ and $I_{c_{1}}(A) = -2 <c_{1}(M), A>$, for $A \in \pi_{2}(M)$ and take
\[
\Gamma = \frac{\pi_{2}(M)}{kerI_{\omega} \cap ker I_{c_{1}}}
\]
We can then form the upwards and downwards Novikov rings by taking a kind of completions of the group $\Gamma$.
Let $\Lambda ^{\downarrow}_{\omega}$ denote the downward Novikov ring, which is defined to be
\[
\Lambda ^{\downarrow}_{\omega} = \{ \sum _{A \in \Gamma} a_{A}e^{A}  \, | \, a_{A} \in \Q, \# \{A \, | \, a_{A} \neq 0, \, I_{\omega}(A) > c \} < \infty \, \, \forall c \in \R \}.
\]
By tensoring the downwards Novikov ring with the homology groups $H_{\ast}(M)$ we can define the \textit{quantum homology} to be
$HQ_{\ast}(M) = H_{\ast}(M) \otimes _{\F} \Lambda ^{\downarrow}_{\omega}$.
The degree of the generator $x_{A} \otimes e^{A}$, which we denote simply by $x_{A}e^{A}$ for notational convenience, is given by $deg(x_{A} e^{A}) = deg(x_{A}) + I_{c_{1}}(A)$.
 \no We also define for any $x \in HQ_{\ast}(M)$ with $x = \sum _{A} x _{A} e^{A}$ the valuation map
$I^{h}_{\omega}(x) = max \{ I_{\omega}(A) \, | \, x_{A} \neq 0 \}$.
\

Define $\Lambda ^{\uparrow}_{\omega}$ to be the upwards Novikov ring, which is defined to be
\[
\Lambda ^{\uparrow}_{\omega} = \{ \sum _{A \in \Gamma} a_{A} e^{A} \, | \, a_{A} \in \Q, \# \{A \, | \, a_{A} \neq 0, \, I_{\omega}(A) < c \} < \infty \, \, \forall c \in \R \}
\]
When we tensor the upwards Novikov ring with the cohomology groups $H^{\ast}(M)$ we can define the \textit{quantum cohomology} to be
$HQ^{\ast}(M) = H^{\ast}(M) \otimes _{\F} \Lambda ^{\uparrow}_{\omega}$.
The degree of the generator $\alpha _{A} e^{A}$ is given by $deg(\alpha _{A} e^{A}) = deg(\alpha) + I_{c_{1}}(A)$.
Also, for any $\alpha \in HQ^{\ast}(M)$ with $\alpha = \sum _{A} \alpha_{A}e^{A}$ we define
$I^{c}_{\omega}(\alpha) = min \{ I_{\omega}(A) \, | \, \alpha _{A} \neq 0 \}$.
\

There is also a product structure defined on both the quantum homology and quantum cohomology that involves Gromov-Witten invariants.  Since we are primarily interested in the cohomology we will just present an outline for this case. The details for the quantum homology case as well as a detailed presentation of Gromov-Witten invariants can be found in \cite{mcduffJHolomorphic}.  Let $\alpha \in H^{k}(M)$, $\beta \in H^{l}(M)$, then the \textit{quantum cup product} of $\alpha$ with $\beta$ is given by
$$\alpha \ast \beta = \sum_{A} (\alpha \ast \beta)_{A} e^{A} ,$$
\no where $deg(\alpha \ast \beta) = deg(\alpha) + deg(\beta)$ and each of the cohomology classes $(\alpha \ast \beta)_{A} \in H^{k + l - 2 c_{1}(A)}(M)$ are defined by the
Gromov-Witten invariants $GW^{M}_{A \, , \, 3}$.  The invariants $GW^{M}_{A \, , \, 3}$ satisfy
$$\int_{c} (\alpha \ast \beta)_{A} = \int_{M} (\alpha \ast \beta)_{A} \cup \eta = GW^{M}_{A \, , \, 3}(a,b,c) ,$$
\no where $c \in H_{k + l - 2c_{1}(A)}(M)$, $a = PD(\alpha)$ \footnote{Here, and in throughout the rest of the paper, the notation ``$PD$'' stands for the Poincar\'e dual.}, $b = PD(\beta)$, $c = PD(\eta)$ and
$deg(a) + deg(b) + deg(c) = 4n - 2 c_{1}(A).$  When this degree condition is not met, then
$GW^{M}_{A \, , \, 3}(a,b,c) = 0$.  Also, when $c_{1}(A) = 0$ then $(\alpha \ast \beta)_{A}$ reduces to the cup product $\alpha \cup \beta$.
\

We are also interested in establishing a form of Poincar\'e duality between the quantum cohomology and homology in a similar manner to what was done in \cite{ohConstruction}.
We can construct an isomorphism between the quantum homology and quantum cohomology by the following map
\[
\flat: HQ^{\ast}(M) \rightarrow HQ_{\ast}(M), \, \, \text{where} \, \, \sum_{A} \alpha _{A}e^{-A} \, \mapsto \, \sum_{A} PD( \alpha _{A})e^{A}
\]
\no and has inverse
\[
\# : HQ_{\ast}(M) \rightarrow HQ^{\ast}(M), \, \, \text{where} \, \, \sum _{A} x_{A} e^{A} \, \mapsto \, \sum _{A} PD(x_{A})e^{-A}.
\]

\no There is also the following relationship between $I^{c}_{\omega}$ and $I^{h}_{\omega}$ given by
$I^{h}_{\omega}(x) = - I^{c}_{\omega}(PD(x))$ for $x \in HQ_{\ast}(M)$, where we denote $PD(x) = \sum _{A} PD(x_{A}) e^{-A}$.

\subsection{The classical Ljusternik--Schnirelman theory: critical value selectors and action selectors}
In order to prove Theorems ~\ref{thm:aspherical} and \ref{thm:weakly} we will use tools from the
Ljusternik-Schnirelmann theory known as critical value selectors and action selectors.  The action selectors, also
known as spectral invariants in the literature, are the Floer theoretic version of critical value selectors.

\begin{Definition}[Critical Value Selectors]
\

Let $M$ be a $n$-dimensional manifold and $f \in C^{\infty}(M)$.  For any $u \in H_{\ast}(M)$ we define
the \textit{critical value selector} by the formula

\begin{align*}
c^{LS}_{u}(f) &= inf \{a \in \R | u \in im(i^{a}) \} \\
              &= inf \{a \in \R | j^{a}(u)  = 0   \},
\end{align*}

\no where $i^{a}:H_{\ast}(\{x \in M | f(x) \leq a \}) \rightarrow H_{\ast}(M)$ and
$j^{a}: H_{\ast}(M) \rightarrow H_{\ast}(M, \{ x \in M | f(x) \leq a \})$ are the natural ``inclusion'' and ``quotient'' maps respectively.
\end{Definition}

\begin{figure}
 \def\svgwidth{.36 \columnwidth}
  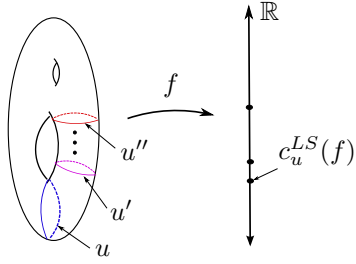 
  \caption{Critical value selector}
  \label{fig:critical}
\end{figure}

One can think of the critical value selectors geometrically in terms of minimax principles.  Take a nonzero
homology class $u \in H_{\ast}(M)$, then one can think of $c^{LS}_{u}(f)$ to be the maximum value $f$ takes on
any representative cycle $u'$ that has been ``pushed down'' as far as possible within
the manifold $M$, see Figure \ref{fig:critical}.  So, when $f$ is a Morse function then we can write
$$c^{LS}_{u}(f) = min \, \ max_{[u'] = u} \{ f(x) \, | \, x \in u' \}.$$
\

The following is a listing of some useful properties concerning critical value selectors.
\begin{itemize}
\item By definition, $c^{LS}_{0}(f) = - \infty$. When $f \equiv \text{const}$ then
         $c^{LS}_{u}(f) \equiv \text{const}$ as well, and for any nonzero $\lambda \in \F$, $c^{LS}_{\lambda u}(f) = c^{LS}_{u}(f)$. For any function $f$ we have
         $$c^{LS}_{1}(f) = min(f) \leq c^{LS}_{u}(f) \leq max(f) = c^{LS}_{[M]}(f).$$

\item Continuity: $c^{LS}_{u}(f)$ is Lipschitz with respect to the $C^{0}$-topology.
\item Triangle Inequality: $c^{LS}_{u \cap w}(f + g) \leq c^{LS}_{u}(g) + c^{LS}_{w}(g)$.
\item Criticality or minimax principle: $c^{LS}_{u}(f)$ is a critical value of $f$.
\item $c^{LS}_{u \cap w}(f) \leq c^{LS}_{u}(f)$, also, if $w \neq [M]$ and the critical points of $f$ are
         isolated, we have strict inequality $c^{LS}_{u \cap w}(f) < c^{LS}_{u}(f)$.
\end{itemize}

\subsubsection{The Hamiltonian Ljusternik--Schnirelman theory: action selectors}
In this section we present the definition and outline the fundamental properties pertaining to action selectors
on cohomology.  The action selectors are defined in a somewhat similar manner, where one big difference is the function $f: M \rightarrow \R$ is replaced by the action functional $\Aa _{H}$ for some Hamiltonian $H$.  There are
numerous sources on the subject of spectral invariants.  Some of the first instances concerning the theory
can be found in \cite{hoferSymplectic}, \cite{viterboSymplectic}.  A thorough treatment of the symplectically aspherical case can be found in \cite{schwarzAction}.  Other known sources can be found
in \cite{hoferSymplectic}, \cite{viterboSymplectic},
\cite{entovCalabi}, \cite{entovRigid}, \cite{ginzburgWeinstein}, \cite{ginzburgAction}, \cite{mcduffJHolomorphic}.  In our paper we will be primarily following the definitions and results found in \cite{ohConstruction}.

\begin{Definition}[Action Selectors on Cohomology]
For any nonzero element $\alpha \in HQ^{\ast}(M) \cong HF^{\ast}(H)$ we define the
\textit{action selector on cohomology} by the formula
\begin{align*}
c^{\alpha}(H) &= inf \{a \in \R - \Ss(H) | PD(\alpha) \in im(i^{a}_{\ast}) \} \\
              &= inf \{a \in \R - \Ss(H) | j^{a}_{\ast}(PD(\alpha)) = 0  \},
\end{align*}

\no where $i^{a}_{\ast}: HF^{(- \infty, a)}_{\ast}(H) \rightarrow HF_{\ast}(H)$ and
$j^{a}_{\ast}: HF_{\ast}(H) \rightarrow HF^{(a, \infty)}_{\ast}(H)$ are the ``inclusion'' and
``quotient'' maps respectively.
\end{Definition}

When $H$ is a non-degenerate Hamiltonian we can write
$$c^{\alpha}(H) = \inf_{[\sigma] = a} \Aa _{H}(\sigma),$$
where $a = PD(\alpha)$ and $\Aa _{H}(\sigma) = max \{ \Aa _{H}(\bar{x}) | \sigma_{\bar{x}} \neq 0 \}$
for $\sigma = \sum \sigma_{\bar{x}} \bar{x} \in CF_{\ast}(H)$.
Just like critical value selectors, one can formulate a geometrical interpretation of the actions selectors, where
they take the various capped one-periodic orbits representing a particular cohomology class and push the ``energy'' down as far as possible.
\

From the above definitions we point out some of their useful properties.
\begin{itemize}
\item Projective invariance: $c^{\lambda \alpha}(H) = c^{\alpha}(H)$ for any $\lambda \in \Q$, $\lambda \neq 0$.

\item Symplectic invariance: $c^{\alpha}(\phi ^{\ast} H) = c^{\alpha}(H)$ for any symplectic
      diffeomorphism $\phi$.

\item Lipschitz continuous: $c^{\alpha}$ is Lispschitz continuous in the $C^{0}$-topology on the space of Hamiltonians $H$.  In particular, $|c^{\alpha}(H) - c^{\alpha}(K)| \leq \| H - K \|$, where
    $\| \cdot \|$ is the Hofer norm.

\item Triangle inequality: $c^{\alpha \ast \beta}(H \# K) \leq c^{\alpha}(H) + c^{\beta}(K)$.

\item Hamiltonian shift: $c^{\alpha}(H + a(t)) = c^{\alpha}(H) + \int^{1}_{0} a(t)dt$, where
      $a: S^{1} \rightarrow \R$.

\item Homotopy invariance: Let $H$ and $K$ be two Hamiltonians which are homotopic to each other, then we have $c^{\alpha}(H) = c^{\alpha}(K)$, for all $\alpha \in QH^{\ast}(H)$.

\item Quantum shift: $c^{a \alpha}(H) = c^{\alpha}(H) - I^{c}_{\omega}(a)$, where $a \in \Lambda ^{\uparrow}_{\omega}$.

\item Valuation inequality: $c^{\alpha + \beta}(H) \leq max\{ c^{\alpha}(H), c^{\beta}(H) \}$ and, the inequality is strict if $c^{\alpha}(H) \neq c^{\beta}(H)$.

\item Spectrality: When $M$ is a rational manifold and $H$ is a one-periodic Hamiltonian on $M$,
                   then $c^{\alpha}(H) \in \Ss (H)$.
\end{itemize}

Let $\widetilde{\mathcal{H}am}(M, \omega)$ be the universal covering space for the group of Hamiltonian diffeomorphisms $\mathcal{H}am(M, \omega)$.  It is worth mentioning that one can also look at the action selectors $c^{\alpha}$ as functions from $\widetilde{\mathcal{H}am}(M, \omega)$ to the reals (\cite{ohConstruction}).

\begin{Remark}
We also point out that one can define the action selectors on the homology of $M$ for any Hamiltonian $H$.  In the non-degenerate case one can define the action selector on the elements $u \in HQ_{\ast}(M)$ by
$c_{u}(H) = \inf_{[\sigma] = u} \Aa _{H}(\sigma)$ for $\sigma = \sum a_{\bar{x}} \bar{x} \in CF_{\ast}(M)$.  The action selectors on homology also satisfy similar properties to the ones defined on the cohomology.  The details of
which are outlined in \cite{ginzburgAction} and \cite{ohConstruction}.  There is one property in particular which interests us: $c_{u}(H) = c^{LS}_{u}(H)$ for $u \in H_{\ast}(M)$ and for $H$ an autonomous and $C^{2}$-small Hamiltonian.  Also, based on the definitions for action selectors on cohomology and homology we see they share the relationship $c^{\alpha}(H) = c_{PD(\alpha)}(H)$.  Putting these two facts together we end up with $c^{\alpha}(H) = c^{LS}_{PD(\alpha)}(H)$ when $H$ is autonomous and $C^{2}$-small.
\end{Remark}

\subsection{Alexander-Spanier cohomology}
\label{sec:Spanier}
Our last preliminary that needs to be introduced is a version of cohomology due to J.M. Alexander and E.H. Spanier.
We will be primarily following the exposition given in \cite{hoferSymplectic}, \cite{masseyTopology}, \cite{spanierTopology}.
\

Begin by fixing a subspace $A \subset M$ and define $\mathcal{O}_{A}$ to be the set of all open neighborhoods of the subset $A$.  One is then able to define an ordered structure on this set in the following manner: for $U, V \in \mathcal{O}_{A}$ we say $U \leq V$ if and only if $V \subseteq U$.  We call $(\mathcal{O}_{A}, \leq)$ the \textit{directed system of neighborhoods for the set $A$}.
\

Now let $\mathcal{C}$ be the category of all subspaces of the manifold $M$ and the category $\mathcal{A}$ to be an algebraic category, which, for our purposes, will either be the category of abelian groups, the category of commutative rings, or the category of modules over a fixed ring.  Define a continuous functor $H: \mathcal{C} \rightarrow \mathcal{A}$ that takes continuous maps $f: V \rightarrow U$, for $U, V \in \mathcal{C}$ and maps
it to a homomorphism $H(f): H(U) \rightarrow H(V)$.  If $U \leq V$ we can define the inclusion map
$i_{VU}: H(U) \rightarrow H(V)$.  From any directed system $\mathcal{O}_{A}$ we define
$D_{A} := \bigoplus_{U \in \mathcal{O}_{A}} H(U)$ and the homomorphism $j_{U}: H(U) \rightarrow D_{A}$ as the inclusion map into the $U$-th component of $D_{A}$.  Next take $K_{A}$ to be the subring that is generated by elements of the form $j_{U}(\alpha_{U}) - j_{V}i_{VU}(\alpha_{U})$ for $U \leq V$, $\alpha_{U} \in H(U)$.  We denote the
quotient of $D_{A}$ by $K_{A}$ by $dir \, \lim_{U \in \mathcal{O}_{A}} H(U) := D_{A} / K_{A}$, which we call
the \textit{direct limit of $A$}.
\

We define $\bar{H}^{\ast}(A ; \Z) := dir \, \lim_{U \in \mathcal{O}_{A}} H^{\ast}(U ; \Z)$ to be the \textit{Alexander-Spanier cohomology} for the subspace $A \subseteq M$.  $H^{\ast}(U ; \Z)$ is the usual singular cohomology.  The restriction maps from $H^{k}(U ; \Z)$ to $H^{k}(A ; \Z)$ end up defining a natural homomorphism from $\bar{H}^{k}(A ; \Z)$ to $H^{k}(A ; \Z)$.  When this homomorphism is an isomorphism that holds for all $k$ and any coefficient group, then we say the subspace $A$ is \textit{taut in $M$}.  The following result gives us a useful list of criteria for when $A$ will be taut in the manifold $M$.

\begin{Theorem}
\label{thm:taut}
In each of the following four cases the subspace $A$ is taut in $M$:
\begin{itemize}
\item $A$ is compact and $M$ is Hausdorff.

\item $A$ is closed and $M$ is paracompact Hausdorff.

\item $A$ is arbitrary and every open subset of $M$ is paracompact Hausdorff.

\item $A$ is a retract of some open subset of $M$.
\end{itemize}
\end{Theorem}

\section{Proofs of Theorems \ref{thm:aspherical} and \ref{thm:weakly}}
\label{sec:proofs}
\

We are now in a position to present the proofs for Theorems ~\ref{thm:aspherical} and \ref{thm:weakly}.  We will begin by showing the result for weakly monotone symplectic manifolds and then present the aspherical one.

\begin{proof}[Proof of Theorem \ref{thm:weakly}]

We start by looking at the fixed points of $\phi_{H}$ which have associated action equal to $c^{\alpha}(H) = a$ and $c^{\alpha \ast \beta}(H) = b$ and call these sets $F_{a}$ and $F_{b}$ respectively.  Let
$\delta > 0$ be small and define $F_{(a - \delta, a + \delta)}$ and $F_{(b- \delta, b + \delta)}$ to be the set of all fixed points of $\phi_{H}$ that have their associated action in the interval $(a - \delta, a + \delta)$ and $(b - \delta, b + \delta)$ respectively.  We then take
$U^{a}_{\delta}$ and $U^{b}_{\delta}$ to be neighborhoods of the sets $F_{(a - \delta, a + \delta)}$ and $F_{(b - \delta, b + \delta)}$ and set $U_{\delta} = U^{a}_{\delta} \cup U^{b}_{\delta}$.  We want to show
$H^{k}(U_{\delta}) \neq 0$ for some $1 \leq k \leq 2n$ and for $\delta$ close to $0$.
\

Suppose not and that $H^{k}(U_{\delta}) = 0$ for all $0 < k \leq 2n$ in order to arrive at a contradiction.
Let $h:M \rightarrow \R$ be a $C^{2}$-small function on $M$ where $h$ is identically equal to zero on the
neighborhoods $U^{a}_{\delta}$ and $U^{b}_{\delta}$ and outside of these sets it is strictly negative.
We can approximate the function
$h$ by a sequence of Morse functions that are at least $C^{2}$-small, call them $h_{n}$, such that
$h_{n} \rightarrow h$ as $n \rightarrow \infty$ in the $C^{0}$-topology and for a fixed $x \in F_{a}$ and $y \in F_{b}$ we have $h_{n}(x) = 0$ and $h_{n}(y) = 0$ only at these points and strictly negative everywhere else.  By making use of the fact that
$c^{LS}_{PD(\eta)}(h_{n}) < 0$ for all $\eta \in H^{k}(M)$, with $k > 0$, and since $c^{\eta}$ is Lipschitz in the
$C^{0}$-topology we have $c^{\eta}(h) < - \delta_{h} < 0$ for all $\eta \in H^{k}(M)$ with $k > 0$ and
$\delta_{h}$ is a positive constant depending on the function $h$.  It is worth noting that we cannot say the same thing about $PD(\alpha)$ because it is possible that $PD(\alpha) = [M]$, which implies $c^{\alpha}(h) = c^{LS}_{[M]}(h) = max(h) = 0$.
\

Define $r: S^{1} \rightarrow \R$ to be a nonnegative, $C^{2}$-small function, equal to zero outside of a small
neighborhood of zero in $S^{1}$.  Set $f_{t} = r(t)h$.  This means that the Hamiltonian flow of $f$ will
be a reparametrization of the flow of $h$ through time $\epsilon = \int^{1}_{0}r(t)dt$.
\


Next we look at the family of Hamiltonians $H \# (s f)$ for $s \in [0,1]$.
By the construction of $f$ we have $H \# (s f) = H$ on the set $U_{\delta}$, but outside of the
set $U_{\delta}$ it is possible, for values of $s$ close to $1$, that $H \# (s f)$ has $1$-periodic
orbits, say $\bar{x}$ and $\bar{y}$, such that $c^{\alpha}(H \# (s f)) = \mathcal{A}_{H \# (s f)}(\overline{x}) \neq a$ or $c^{\alpha \ast \beta}(H \# (sf)) = \mathcal{A}_{H \# (sf)} (\overline{y}) \neq b$.  However, we claim that for small values of $s$ that we can prevent this situation from occurring.  In particular, we claim that one can find a nonzero $s'$ in $[0,1]$ such that for all $0 \leq s \leq s'$ the Hamiltonians
$H \# (s f)$ may have new $1$-periodic orbits such that their action is not in $\mathcal{S}(H)$ and that their
values may drift into the interval $(a - \delta, a + \delta)$, but by picking $s'$ small enough these
new critical values for $H \# (s' f)$ cannot drift into the neighborhoods
$(a - \frac{\delta}{2}, a + \frac{\delta}{2})$ and $(b - \frac{\delta}{2}, b + \frac{\delta}{2})$.  We will show this fact below in Lemma ~\ref{lemma:existence} and
suppose for the time being that such an $s'$ exists.  Then for all $0 \leq s \leq s'$ we have
$\mathcal{S}( H \# (s f) ) \cap (a - \frac{\delta}{2}, a + \frac{\delta}{2}) =
\mathcal{S}(H) \cap (a - \frac{\delta}{2}, a + \frac{\delta}{2})$ and $\mathcal{S}( H \# (s f) ) \cap (b - \frac{\delta}{2}, b + \frac{\delta}{2}) =
\mathcal{S}(H) \cap (b - \frac{\delta}{2}, b + \frac{\delta}{2})$.
\

Now, when $h$ and $r$ are sufficiently $C^{2}$- small, $\epsilon h$ and $f$ have the same periodic
orbits, which are the critical point of $h$, and they have the same action spectrum.
This is also true for the functions $\epsilon s' h$ and $s' f$.
The same will be true for every function in the linear family $\tilde{f}_{l} = (1-l)\epsilon s' h + l s' f$, with $l \in [0,1]$, connecting $\epsilon s' h$ and $s' f$.
Using the continuity property of $c_{u}$, the fact that each $\mathcal{S}(\tilde{f}_{l})$ is a set of measure zero,
and that $\mathcal{S}(\tilde{f}) = \mathcal{S}(\tilde{f}_{l})$ for all $l$, we conclude that
$c_{PD(\beta _{A})}(s' f) = c_{PD(\beta _{A})}(\epsilon s' h) < 0$ for each $PD(\beta _{A})$ in $PD(\beta)$.
\

We again use the continuity property
of $c^{\alpha}$ and that the sets $\mathcal{S}(H \# (sf))$ have measure zero for all $s$ to give us
$c^{\alpha}(H \# (s' f)) = c^{\alpha}(H) = a$.  By the construction of $H \# (s'f)$, we also have
$c^{\alpha \ast \beta}(H \# (s' f)) = c^{\alpha \ast \beta}(H)$ as well.  We then use the following triangle inequality for action selectors to give
\begin{align*}
c^{\alpha \ast \beta}(H) &= c^{\alpha \ast \beta}(H \# (s' f)) \leq c^{\alpha}(H) + c^{\beta}(s'f) \\
             & = c^{\alpha}(H) + c_{PD(\beta)}(s'f) \\
              & = c^{\alpha}(H) + c_{\sum _{A} PD(
              \beta_{A}) e^{A}}(s'f) \leq c^{\alpha}(H) + \max \{ c_{PD(\beta_{A}) e^{A}}(s'f) \, | \, PD(\beta_{A})\neq 0 \} \\
 &= c^{\alpha}(H) + c_{PD(\beta_{A}) e^{A}}(s'f) \\
 & = c^{\alpha}(H) + c^{LS}_{PD(\beta_{A})}(s'f) + I^{h}_{\omega}(e^{A})\\
 & < c^{\alpha}(H) + I^{h}_{\omega}(e^{A}) \\
 & \leq c^{\alpha}(H) + I^{h}_{\omega}(PD(\beta)) \\
 & = c^{\alpha}(H) - I^{c}_{\omega}(\beta).
\end{align*}
\no and creates a contradiction to the fact that $c^{\alpha \ast \beta}(H) = c^{\alpha}(H) - I^{c}_{\omega}(\beta)$.  This implies that $H^{k}(U_{\delta}) \neq 0$ for $k = deg(\beta _{A})$

\begin{Remark}
A quick observation about the $\max \{ c_{PD(\beta_{A}) e^{A}}(s'f) \, | \, PD(\beta_{A})\neq 0 \}$ term above.  Technically we should write $sup$, but since we have the identity
$c_{PD(\beta_{A})e^{A}}(s'f) = c^{LS}_{PD(\beta_{A})}(s'f) + I^{h}_{\omega}(e^{A})$, $M$ is a compact manifold, and because we are working with the upwards Novikov ring (along with the $I^{h}_{\omega}$ valuation) we will actually end up with the $sup$ being one of the $c_{PD(\beta_{A})e^{A}}(s'f)$ terms.
This means that we can actually look at the $max$ rather than the $sup$ in the above string of inequalities and also for notational simplicity we just labeled this maximum to be $c_{PD(\beta_{A}) e^{A}}(s'f)$.
\end{Remark}
\

Now let $\mathcal{O}_{F'}$ be a directed system of neighborhoods for the set $F' = F_{a} \cup F_{b}$.  Then
Theorem ~\ref{thm:taut} along with the basic properties outlined in Section ~\ref{sec:Spanier} almost immediately implies that $\beta_{| F} \neq 0$ in $HQ^{\ast}(F)$ and that
$H^{k}(F) \neq 0$ for $k = deg(\beta_{A})$, which proves Theorem ~\ref{thm:weakly}.
\end{proof}
\

With the above in mind, we are able to prove Theorem ~\ref{thm:aspherical}.

\begin{proof}[Proof of Theorem \ref{thm:aspherical}]

First recall the assumption that $M$ is symplectically aspherical, let $CL(M) = m$ and let $\alpha_{1}, \cdots, \alpha_{m}$ be cuplength representative in $H^{\ast > 0}(M)$.
Using a result from \cite{ginzburgAction} which in the symplectically aspherical case says that for $\alpha, \beta \in H^{\ast}(M)$ with $deg(\beta) > 0$
we have $c^{\alpha \cup \beta}(H) \leq c^{\alpha}(H)$.  This gives us the following monotonically decreasing sequence
$$c^{\tilde{\alpha}_{m}}(H) \leq \cdots \leq c^{\tilde{\alpha}_{1}}(H) \leq c^{\tilde{\alpha}}(H)$$
\no with
$$\tilde{\alpha} = PD([M]), \, \, \tilde{\alpha}_{1} = \alpha_{1}, \, \tilde{\alpha}_{2} = \alpha_{2} \cup \tilde{\alpha}_{1}, \ldots , \, \tilde{\alpha}_{m} = \alpha_{m} \cup \tilde{\alpha}_{m-1}. $$
\no Since each $c^{\beta}(H) \in \Ss (H)$ and $\# \Ss (H) \leq m$ it implies there must be equality somewhere in the
above chain of inequalities.  So, $c^{\tilde{\alpha}_{i + 1}}(H) = c^{\tilde{\alpha}_{i}}(H)$
for some $1 \leq i \leq m$, or $\tilde{\alpha}_{i} = PD([M])$.
Since $\tilde{\alpha}_{i + 1} = \alpha_{i + 1} \cup \tilde{\alpha}_{i}$ we just rename $\alpha_{i + 1} = \beta$ and
$\tilde{\alpha_{i}} = \alpha$ for notational convenience.  This means
$c^{\alpha \cup \beta}(H) = c^{\alpha}(H)$ and as we have pointed out in Section ~\ref{sec:quantum} the quantum
product in the symplectically aspherical case reduces to the cup product, i.e. $\alpha \ast \beta = \alpha \cup \beta$, so we can apply Theorem ~\ref{thm:weakly}, which immediately gives us our result.
\end{proof}

In Theorem \ref{thm:weakly} we needed to show we can find some nonzero $s'$ in the unit interval which satisfies the property that $\mathcal{S}(H \# (s'f))$ does not gain any new critical points within
either of the intervals $(a - \frac{\delta}{2}, a + \frac{\delta}{2})$ or $(b - \frac{\delta}{2}, b + \frac{\delta}{2})$.  We will show that this is true for only a single interval $(a- \frac{\delta}{2}, a + \frac{\delta}{2})$, since the proof generalizes to the case when there are two intervals.

\begin{Lemma}
\label{lemma:existence}
There exists some nonzero $s'$ in $[0,1]$ such that $\mathcal{S}(H \# (s f))$ does not
gain any new critical values within the interval $(a - \frac{\delta}{2}, a + \frac{\delta}{2})$ for all
$0 \leq s \leq s'$.
\end{Lemma}

\begin{proof}
Suppose not and we cannot find such a number $s'$.  This means we can find a sequence of $s_{n}$ in $[0,1]$
where $s_{n} \rightarrow 0$ as $n \rightarrow \infty$ and that there exists a one-periodic orbit $x_{n}$
of $X_{H \# (s_{n}f)}$ such that $\mathcal{A}_{H \# (s_{n})f}(\overline{x}_{n}) = a_{n}$ with
$lim_{n \rightarrow \infty} a_{n} = a$.  Now, since $H \# (s_{n}f) = H$ on the set $U^{a}_{\delta}$, it means
the fixed points for $\phi_{H \# (s_{n}f)}$, with associated action $a_{n} \in \Ss (H \# (s_{n}f))$, can't be elements of the set $U^{a}_{\delta}$.
\

Our next step is to show we can find a one-periodic orbit $x_{\ast}$ for
$X_{H}$ with $\mathcal{A}(\overline{x}_{\ast}) = a$ that comes from some subsequence of the $x_{n}$'s.
In order to show this we will use the generalized Arzela Ascoli theorem for metric spaces which says the following:
If $X_{1}$ is compact Hausdorff space, $X_{2}$ is a metric space, $C(X_{1}, X_{2})$ be the set of continuous functions from $X_{1}$ to $X_{2}$, and let $\{ f_{n} \}$ be a sequence of functions in $C(X_{1}, X_{2})$ that is
uniformly bounded and equicontinuous, then there exists a subsequence $\{ f_{n_{j}} \}$ that converges uniformly.
We apply this to our capped loops $\bar{x}_{n}$, taking $X_{1} = [0,1]$ and $X_{2} = M$.  Let $d$ be the distance function that comes from the Riemannian metric $g$ on $M$. We want to first show that there exists some real number
$L > 0$ such that $d(x_{n}(t), x_{n}(s)) \leq L | t - s |$ for all $n$.  Note that since the manifold $M$
is compact that there is a uniform bound on the $X_{H \#(s_{n}f)}$ where
$\| X_{H \# (s_{n}f)} \| \leq L$ for some $L > 0$ and for all $n$.  Since the distance between two points $p,q \in M$ is given by $d(p,q) = \inf_{\gamma}(L(\gamma))$ for $L(\gamma) = \int^{b}_{a} \| \dot{\gamma}(t) \| dt$ we have
$$ d(x_{n}(t), x_{n}(s)) \leq \int^{t}_{s} \| \dot{x}_{n}(u) \| du = \int^{t}_{s} \| X_{H \# (s_{n}f)}(x_{n}) \| du
\leq L | t - s | .$$
\no This shows that the family of curves $\{ x_{n} \}$ is uniformly Lipschitz, which implies
that this family of curves is uniformly bounded and equicontinuous.  This means there is a subsequence $\{ x_{n_{j}} \}$ that converges to the curve $x_{\ast}$.  The curve $x_{\ast}$ is only a continuous loop from $[0,1]$ to $M$, but we can use the following
result which tells us that $x_{\ast}$ is actually a smooth solution to $X_{H}$.
\

%

\begin{Proposition}
Assume that the sequence of Hamiltonian vector fields $X_{H_{n}} \rightarrow X_{H}$ as $n \rightarrow \infty$ in
the $C^{0}$-topology and $x_{n}$ is a solution to $X_{H_{n}}$ and $x_{n} \rightarrow x_{\ast}$ in the $C^{0}$-topology.
Then $x_{\ast}$ is a solution to $X_{H}$.
\end{Proposition}

This means $x_{\ast}$ is a one-periodic solution to $X_{H}$.  Our next step is to show that
$\mathcal{A}(H)(\overline{x}_{\ast}) = a$.
Let $\epsilon > 0$.  Since $\mathcal{A}_{H \# (s_{n} f)}(\overline{x}_{n}) = a_{n}$ we can find some
$N_{1} $ such that for all $n > N_{1}$ we get
$|a_{n} - a | < \frac{\epsilon}{3}$.  At the same time, the Hamiltonians $H \# (s_{n}f) \rightarrow H$ in the
$C^{1}$-topology and we can find some $N_{2}$ such that for all $n > N_{2}$ we have
$|\mathcal{A}_{H}(\overline{x}) - \mathcal{A}_{H \# (s_{n}f)}(\overline{x}) | < \frac{\epsilon}{3}$.
Lastly, since the $x_{n_{j}}$ converge uniformly to the one-periodic solution $x_{\ast}$ of $X_{H}$ we can find some $N_{3}$ such that for all $n_{j} > N_{3}$ we get that
$|\mathcal{A}_{H}(\overline{x}_{\ast}) - \mathcal{A}_{H}(\overline{x}_{n_{j}})| < \frac{\epsilon}{3}$.  Then for
$N = max \{ N_{1}, N_{2}, N_{3} \}$ we have for $n > N$ that
\begin{align*}
|\mathcal{A}_{H}(\overline{x}) - a| \leq & \, \, | \mathcal{A}_{H}(\overline{x}_{\ast}) - \mathcal{A}_{H}(\overline{x}_{n_{j}})|
+ | \mathcal{A}_{H}(\overline{x}_{n_{j}}) - \mathcal{A}_{H \# (s_{n_{j}}f)}(\overline{x}_{n_{j}}) | + \\
    & \, \, |\mathcal{A}_{H \# (s_{n_{j}}f)}(\overline{x}_{n_{j}}) - a | < \epsilon.
\end{align*}
Since this is true for every $\epsilon > 0$ it gives $\mathcal{A}_{H}(\overline{x}_{\ast}) = a$.
\

Our next step is to show that the fixed point $x_{\ast}(0) = x_{\ast}(1)$ for $\phi_{H}$ that has associated action
$\mathcal{A}_{H}(\overline{x}_{\ast}) = a$ is a point that is outside of the $U^{a}_{\delta}$.  In order to do this we will look at the other fixed points
$p_{n_{j}}$ for $\phi_{H \# (s_{n_{j}}f)}$ that come from the loops $x_{n_{j}}$.  In order to simplify the notation
we will just relabel the points $p_{n_{j}}$ to be $p_{n}$. Now, since $M$ is a compact metric space we know that it is sequentially compact, meaning any sequence $\{ y_{n} \}$ has a convergent subsequence $\{ y_{n_{j}} \}$, and that the
collection of points $\{ p_{n} \}$ has a convergent subsequence $\{ p_{n_{j}} \}$ that converges to the point
$p$.  In fact, the limit point $p$ is a fixed point for $\phi_{H}$, which we will show.
Let $\epsilon > 0$ and we show that $d(\phi_{H}(p),p) < \epsilon$.
Since $s_{n}f \rightarrow 0$ pointwise as $n \rightarrow \infty$ and since $\phi_{H}$ is continuous it implies that
$$\phi_{H \# (s_{n_{j}}f)} = \phi_{H} \circ \phi_{s_{n_{j}}f} \rightarrow \phi_{H}$$
\no pointwise as $j \rightarrow \infty$.  Then there exists some $N_{1}$ such that for all $n_{j} > N_{1}$ we have
$d(\phi_{H}(p), \phi_{H \# (s_{n_{j}}f)}(p)) < \frac{\epsilon}{3}$.  We can also find some $N_{2}$ such that
for all $n_{j} > N_{2}$ that
$$d(\phi_{H \# (s_{n_{j}}f)}(p), \phi_{H \# (s_{n_{j}}f)}(p_{n_{j}})) < \frac{\epsilon}{3}$$
\no and we can find an
$N_{3}$ such that for all $n_{j} > N_{3}$ we get $d(p_{n_{j}}, p) < \frac{\epsilon}{3}$.  For
$n_{j} > N = max \{N_{1}, N_{2}, N_{3} \}$ we end up with
\begin{align*}
d(\phi_{H}(p), p) \leq & \, \, d(\phi_{H}(p), \phi_{H \# (s_{n_{j}}f)}(p)) + \\
                                 & \, \, d(\phi_{H \# (s_{n_{j}}f)}(p) , \phi_{H \# (s_{n_{j}}f)}(p_{n_{j}}) ) + d(p_{n_{j}}, p) < \epsilon.
\end{align*}
\no So, $p$ is a fixed point for $\phi_{H}$.
\

Let $x$ be the loop formed by the curve $\phi^{t}_{H}(p)$ for $0 \leq t \leq 1$.  Since both $x_{\ast}$ and
$x$ are one-periodic solutions for $X_{H}$ and they both have the point $p$ on them, then by uniqueness of
solutions of O.D.E.'s it forces $x_{\ast} = x$.  Then, we end up with $p$ being a fixed point of $\phi_{H}$, which is
on the curve $x$, and has the associated action $\mathcal{A}_{H}(\overline{x}) = a$.  However, we have that
$p$ is the limit point of the points $p_{n_{j}}$ and we know that $p_{n_{j}} \not\in U^{a}_{\delta}$
for all $n$ and means that $p \not\in U^{a}_{\delta}$, which creates a contradiction.
\end{proof}

\medskip
 \bibliographystyle{alpha}  
 \bibliography{hamiltonian-references}


\end{document}